\documentclass[11pt,a4paper]{amsart}
\usepackage[pdftex]{graphicx}
\usepackage{amsfonts}
\usepackage{amssymb}
\usepackage{mathrsfs}
\usepackage{stmaryrd}
\usepackage{amsmath}
\usepackage{amsthm}
\usepackage{enumerate}
\usepackage{nomencl}
\usepackage[bookmarks=true]{hyperref}   
\usepackage{esint}
\usepackage{dsfont}
\usepackage{upgreek}

\hypersetup{ colorlinks}

\title[Quadratic estimates for perturbed Dirac type operators]
	{Quadratic estimates for perturbed Dirac type operators on doubling measure metric spaces}
\author{Lashi Bandara}
\address{Lashi Bandara, Centre for Mathematics and its Applications, 
Australian National University, Canberra, ACT, 0200, Australia}
\urladdr{\href{http://maths.anu.edu.au/~bandara}{http://maths.anu.edu.au/~bandara}}
\email{\href{mailto:lashi.bandara@anu.edu.au}{lashi.bandara@anu.edu.au}}

\newtheorem{theorem}{Theorem}[section]
\newtheorem{corollary}[theorem]{Corollary}
\newtheorem{lemma}[theorem]{Lemma}
\newtheorem{proposition}[theorem]{Proposition}

\newtheorem{definition}[theorem]{Definition}

\newcommand{\mdot}{\cdotp}

\newcommand{\cbrac}[1]{\left(#1\right)}

\newcommand{\dbrac}[1]{\left\{#1\right\}}
\newcommand{\modulus}[1]{\left|#1\right|}
\newcommand{\set}[1]{\dbrac{#1}}

\newcommand{\dom}{ {\mathcal{D}}}
\newcommand{\ran}{ {\mathcal{R}}}
\newcommand{\nul}{ {\mathcal{N}}}
\DeclareMathOperator{\card}{card}

\newcommand{\R}{\mathbb{R}}
\newcommand{\C}{\mathbb{C}}
\newcommand{\In}{\mathbb{Z}}

\newcommand{\Na}{\ensuremath{\mathbb{N}}}

\newcommand{\script}[1]{\mathscr{#1}}

\newcommand{\re}{{\rm Re}\, }			

\renewcommand{\emptyset}{\varnothing}
\newcommand{\union}{\cup}
\newcommand{\Union}{\bigcup}
\newcommand{\intersect}{\cap}

\newcommand{\close}[1]{\overline{#1}}		

\newcommand{\ind}[1]{\raisebox{\depth}{\(\chi\)}_{#1}}	

\renewcommand{\epsilon}{\varepsilon}
\renewcommand{\phi}{\varphi}

\newcommand{\norm}[1]{\left\| #1 \right\|}			
\newcommand{\spt}[1]{{\rm spt} {\text{ }}#1}	

\DeclareMathOperator{\dist}{dist}		

\DeclareMathOperator{\diam}{diam}		

\DeclareMathOperator{\len}{\ell}			

\DeclareMathOperator{\divv}{div}		

\newcommand{\bnd}{\partial}			
\DeclareFontFamily{OT1}{restrictfont}{}
\DeclareFontShape{OT1}{restrictfont}{m}{n}{<-> fmvr8x}{}

\newcommand{\adj}[1]{{#1}^\ast}			

\newcommand{\inprod}[1]{\left\langle #1 \right\rangle}	
\newcommand{\grad}{\nabla}			
\DeclareMathOperator{\Lip}{\bf Lip}			
\DeclareMathOperator{\Lipp}{Lip}		

\newcommand{\bddlf}{\mathcal{L}} 	

\newcommand{\Lp}[2][{}]{{\rm L}^{#2}_{\rm #1}}		

\newcommand{\Max}{{\mathcal{M}}}			
\newcommand{\Cone}{\Gamma}				
\newcommand{\Tent}{\mathrm{T}}					

\newcommand{\Hil}{\script{H}}			

\newcommand{\Q}[1][{}]{Q_{#1}}			
\newcommand{\DyQ}{\script{Q}}			

\newcommand{\p}{p}				

\newcommand{\Spa}{\mathcal{X}}				

\newcommand{\sC}{\script{C}}
\newcommand{\sE}{\script{E}}

\newcommand{\cE}{\mathcal{E}}

\newcommand{\cF}{\mathcal{F}}

\newcommand{\Poincare}{Poincar\'e~}		

\newcommand{\Av}{\mathcal{A}}
\newcommand{\pri}{\upgamma}

\newcommand{\Mul}{\mathrm{M}}

\newcommand{\Nfs}{\mathcal{N}}
\newcommand{\Carl}{\mathcal{C}}
\newcommand{\Cf}{\mathrm{C}}
\newcommand{\CBox}{\mathrm{R}}

\begin{document}

\begin{abstract}
We consider 
perturbations of Dirac type operators
on complete, connected
metric spaces
equipped
with a doubling measure. 
Under a suitable set of
assumptions, we prove
quadratic estimates
for such operators
and hence deduce
that these operators have
a bounded functional calculus.
In particular, we deduce a Kato square
root type estimate.
\end{abstract}

\maketitle

\tableofcontents

\parindent0cm
\setlength{\parskip}{\baselineskip}

\section{Introduction}

Let $\Spa$ be a complete, connected
metric space and
$\mu$ a Borel-regular doubling measure.
We consider  densely-defined, 
closed, nilpotent operators
$\Gamma$ on $\Lp{2}(\Spa, \C^N)$
and perturbed Dirac type operators
$\Pi_B = \Gamma + B_1 \adj{\Gamma} B_2$,
where $B_i$ are strictly accretive
$\Lp{\infty}$ matrix valued functions.
We prove quadratic estimates
$$
\int_{0}^\infty \norm{t\Pi_B(1 + t^2\Pi_B)^{-1}u}^2\ \frac{dt}{t}
	\simeq \norm{u}^2$$
for $u \in \close{\ran(\Pi_B)}$
under a set of hypotheses
(H1)-(H8).
These estimates are equivalent
to $\Pi_B$ having a bounded
holomorphic functional 
calculus.
This allows us to conclude
that $\dom(\sqrt{\Pi_B^2}) = \dom(\Pi_B)
= \dom(\Gamma) \intersect \dom(B_1 \adj{\Gamma} B_2)$
and that $\norm{\sqrt{\Pi_B^2}u} \simeq \norm{\Pi_B u}
\simeq \norm{\Gamma u} + \norm{B_1 \adj{\Gamma} B_2u}$.
When
$\Spa = \R^n$ and 
$\mu$ is the Lebesgue
measure, it is shown 
by Axelsson, Keith 
and McIntosh in \cite{AKMC} that this implies
$\dom(\sqrt{-\divv A \grad})) = \dom(\grad)$
and $\norm{\sqrt{-\divv A \grad}u} \simeq \norm{\grad u}$
for an appropriate class of perturbations $A$. 
Thus, we are justified in calling this a 
\emph{Kato square root type estimate}.

We proceed to prove our theorem
based on the ideas 
presented in \cite{AKMC}.
These ideas date back to the 
resolution of the Kato conjecture
by Auscher, Hofmann, Lacey,
McIntosh and Tchamitchian
in \cite{AHLMcT}.
The exposition \cite{HofmannEx} 
by Hofmann
is an excellent survey
of the history and resolution 
of the Kato conjecture.
Further historical references
include the article \cite{Mc90} by McIntosh
and \cite{AT} by Auscher and Tchamitchian.
More recently, the proof in \cite{AKMC}
was generalised by
Morris in \cite{Morris} for 
complete Riemannian manifolds
with \emph{exponential volume
growth}.
This work is beneficial
to us
since we rely
upon the same
abstract dyadic decomposition 
of Christ in \cite{Christ}.

The main novelty 
of the work presented here
is that we
have separated
the assumptions on the
operator $\Gamma$ from 
the underlying differentiable structure
of the space.
In general,
the spaces we consider may not admit
a differentiable structure.
However, we are motivated
by the existence of
measure metric spaces
more general than Riemannian
manifolds admitting such
structures.
See the work of Cheeger 
in \cite{Cheeger} and of Keith
in \cite{Keith}.

In our exposition, we follow the 
structure of the proof in
\cite{AKMC}.
We rephrase the proof
purely in terms of Lipschitz functions.
We use an upper gradient quantity,
namely the \emph{pointwise Lipschitz constant},
as a replacement for a gradient.
This is the key feature
that allows us to generalise
the proof in \cite{AKMC}.

The structure of this paper is as follows.
In \S\ref{HypoRes}, we state the hypotheses (H1)-(H8)
under which we obtain the quadratic estimates
and state the main results. We
devote \S\ref{Dya} to illustrating some important
consequences of the dyadic
decomposition in \cite{Christ}.
In \S\ref{MaxCarl}, 
we present some results about Carleson measures
and maximal functions on
doubling measure metric spaces.
These tools are crucial since
the proof of the main result proceeds
by reducing the main estimate 
to a Carleson measure estimate.
Lastly, we give a proof of the main
theorem in \S\ref{HarmAnal},
taking care to avoid 
unnecessary repetition
of the work of \cite{AKMC}
and \cite{Morris}, and 
highlight the key differences
which we have introduced.

\section*{Acknowledgements}
\addtocontents{Acknowledgements}

This work was undertaken
at the Centre for Mathematics
and its Applications at the
Australian National University
and supported 
by this institution and 
an Australian Postgraduate Award.

I am indebted to 
Alan McIntosh for his 
support, insight, and
excellent supervision
which made this work possible.
I would also like to thank
Andrew Morris, Rob Taggart,
and Pierre Portal for their 
encouragement and helpful 
suggestions.

\section{Hypotheses and the main results}
\label{HypoRes}

We list a set of
hypotheses (H1)-(H8).
These assumptions
are similar those in
\cite{AKMC},
with the exception of (H6) 
and (H8) which
require
modification due to 
the lack of a differentiable structure
in our setting. 
The assumptions (H1)-(H3)
are purely operator theoretic
and thus hold in sufficient generality. 
They are taken in verbatim from \cite{AKMC}
but we list them here for completeness. 
We emphasise that here,
$\Hil$ denotes an abstract Hilbert space.

\begin{enumerate}
\item[(H1)]
	The operator $\Gamma: \dom(\Gamma) \to \Hil$
	is closed, densely-defined and \emph{nilpotent}
	($\Gamma^2 = 0$).

\item[(H2)]
	The operators $B_1, B_2 \in \bddlf(\Hil)$
	satisfy
	\begin{align*}
	&\re\inprod{B_1u,u} \geq \kappa_1 \norm{u} &&\text{whenever }u \in \ran(\adj{\Gamma}), \\
	&\re\inprod{B_2u,u} \geq \kappa_2 \norm{u} &&\text{whenever }u \in \ran(\Gamma)
	\end{align*}
	where $\kappa_1, \kappa_2 > 0$ are constants.

\item[(H3)]
	The operators $B_1, B_2$ satisfy
	$$B_1B_2 (\ran(\Gamma)) \subset \nul(\Gamma)\text{ and }B_2B_1(\ran(\adj{\Gamma})) \subset \nul(\adj{\Gamma}).$$
\end{enumerate}

The full implications of these
assumptions are
listed in \S4 in \cite{AKMC}.
However, for the sake
of convenience,
we include some relevant details
from this reference.
Define $\adj{\Gamma}_B = B_1 \adj{\Gamma} B_2$,
$\Pi_B = \Gamma + \adj{\Gamma}_B$
and $\Pi = \Gamma + \adj{\Gamma}$.
Furthermore, define the following
associated bounded operators:
\begin{multline*}
R_t^B = (1 + it \Pi_B)^{-1},\ 
P_t^B = (1 + t^2 \Pi_B^2)^{-1},\\
Q_t^B = t\Pi_B (1 + t^2 \Pi_B^2)^{-1},\ 
\Theta_t^B = t \adj{\Gamma}_B (1 + t^2 \Pi_B^2)^{-1},
\end{multline*}
and write $R_t, P_t, Q_t,
\Theta_t$ by setting $B_1 = B_2 = 1$.
With this in mind, we
bring the attention
of the reader 
to the following important
proposition.

\begin{proposition}[Proposition 4.8 of \cite{AKMC}]
\label{Prop:HypRes:Main}
Suppose that $(\Gamma, B_1, B_2)$ satisfy
the hypotheses (H1)-(H3) and that
there exists $c > 0$ such that
$$ \int_{0}^\infty \norm{\Theta_t^B P_t u}^2\ \frac{dt}{t}
	\leq c \norm{u}^2$$
for all $u \in \ran(\Gamma)$,
together with three similar estimates
obtained by replacing $(\Gamma, B_1, B_2)$
by $(\adj{\Gamma},B_2,B_1)$,
$(\adj{\Gamma},\adj{B_2}, \adj{B_1})$
and $(\Gamma, \adj{B_1}, \adj{B_2})$.
Then, $\Pi_B$ satisfies
$$
\int_{0}^\infty \norm{Q_t^B u}^2\ \frac{dt}{t}
	\simeq \norm{u}^2$$ 
for all $u \in \close{\ran(\Pi_B)} \subset \Hil$.
Thus, $\Pi_B$ has a bounded  $H^\infty$ 
functional calculus.
\end{proposition}

For a fuller treatment of the theory 
of sectorial operators and
holomorphic functional calculi,
see \cite{ADMc} by Albrecht, Duong and McIntosh, and 
\cite{Kato} by Kato.
Furthermore, Morris deals with 
local quadratic estimates
and their functional calculus
implications in \cite{MorrisLocal}.

It is the conclusion of
the above proposition that
is our primary objective.
We note 
as do the authors
of \cite{AKMC} that we require additional assumptions
on $\Spa$ and $(\Gamma, B_1,B_2)$ 
in order to satisfy the hypothesis of the proposition.
Thus, we start with the following definition.
\begin{definition}[Doubling measure]
We say that $\mu$ is a doubling measure on $\Spa$
if there exists a constant $C_D \geq 1$ such that
for all $x \in \Spa$ and $r > 0$, 
$$
0 < \mu(B(x,2r)) \leq C_D \mu(B(x,r)) < \infty.$$
We call $C_D$ the doubling constant and
we let $\p = \log_2(C_D)$.
\end{definition}

It is, in fact, easy to show
that a measure is doubling if and only if
$\mu(B(x,\kappa r)) \leq C_D \kappa^p \mu(B(x,r))$
whenever $\kappa > 1$.

We are now in a position
to list (H4) and (H5).
 
\begin{enumerate}
\item[(H4)] 
	Let $\Spa$ be a complete, connected
	metric space 
	and $\mu$ a  Borel-regular 
	measure on $\Spa$
	that is doubling.
	Then set $\Hil = \Lp{2}(\Spa, \C^N; d\mu)$.

\item[(H5)] $B_i \in L^\infty(\Spa, \bddlf(\C^N))$
	for $i = 1,2$.
\end{enumerate}

For convenience, we write $\Hil = \Lp{2}(\Spa)$
or $\Lp{2}(\Spa,\C^N)$.

Note
that the two hypotheses
above are the obvious
adaptations of (H4) and
(H5) in \cite{AKMC}.
The matter of (H6) is a
little more complicated
since (H6) of \cite{AKMC} and \cite{Morris}
involves $\grad$ which
in general does not exist for us.
To circumvent this obstacle,
we define the following
quantity.
\begin{definition}[Pointwise Lipschitz constant]
\index{Pointwise Lipschitz constant}
For $\xi: \Spa \to \C^N$ Lipschitz, define
$\Lipp \xi: \Spa \to \R$ by
$$\Lipp \xi(x) = \limsup_{y \to x} \frac{\modulus{\xi(x) - \xi(y)}}{d(x,y)}.$$
We take the convention that
$\Lipp \xi (x) = 0$ when $x$ is an isolated point.
\end{definition}

Letting $\Lip \xi$ denote the
\emph{Lipschitz constant} of $\xi$, we
note that by construction,
$\Lipp \xi(x) \leq \Lip \xi$
for all $x \in \Spa$. 
Also, $\Lipp \xi$ is a Borel function
and therefore measurable.
Many of the properties of $\Lipp\xi$ 
are described in greater detail 
in \cite{Cheeger}. We note that 
it is from this reference that
we 
have borrowed this notation and the term 
``pointwise Lipschitz constant.''
 
\begin{enumerate}
\item[(H6)]
	For every bounded Lipschitz function $\xi: \Spa \to \C$,
	multiplication by $\xi$ preserves
	$\dom(\Gamma)$ and 
	$\Mul_{\xi} = [\Gamma, \xi I]$
	is a multiplication operator.
	Furthermore, there exists a constant $m > 0$ such that
	$\modulus{\Mul_{\xi} (x)} \leq m \modulus{\Lipp{\xi}(x)}$
	for almost all $x \in \Spa$.
\end{enumerate}

We note that this implies
the same hypothesis when $\Gamma$
is replaced by $\adj{\Gamma}$ and $\Pi$.
This observation is made 
in  \cite{Morris}
and originated in \cite{AKM2}.

When
$\Spa = \R^n$ and $\mu$
is the Lebesgue measure
(the setting in \cite{AKMC}),
our (H6) is automatically
satisfied since  
$\modulus{\grad \xi(x)} = \modulus{\Lipp \xi(x)}$
for almost all $x \in \R^n$.

The following is called the \emph{cancellation hypothesis}.
In the work of \cite{Morris} and \cite{AKM2},
this hypothesis is replaced by a weaker estimate
which is applicable for local quadratic estimates
as described by Morris in \cite{MorrisLocal}. The
estimates we require are global and thus
we assume the cancellation hypothesis in \cite{AKMC}.
We denote the \emph{support} of a function $f$
by $\spt f$.

\begin{enumerate}
\item[(H7)]
	For each open ball $B$, we have
	$$ \int_{B} \Gamma u\ d\mu = 0 
	\quad\text{and}\quad
	\int_{B} \adj{\Gamma} v\ d\mu = 0$$ 
	for all $ u \in \dom(\Gamma)$ with $\spt u \subset B$ 
	and for all $v  \in \dom(\adj{\Gamma})$ with $\spt v \subset B$.
\end{enumerate}

The last assumption is a \emph{\Poincare hypothesis}.
In \cite{Morris}, a \Poincare inequality on balls is assumed
as a separate hypothesis. Their (H8) is
a coercivity assumption
following \cite{AKMC}.
In our work, we find that
a \Poincare type hypothesis
with respect to the unperturbed
operator $\Pi$
is a sensible substitution. 

\begin{enumerate}
\item[(H8)]
There exists $C' > 0$ and $c > 0$ 
such that for all balls $B=B(y,r)$ 
$$
\int_{B} \modulus{u(x) - u_B}^2\ d\mu(x)
	\leq C' r^2 \int_{cB} \modulus{\Pi u(x)}^2\ d\mu(x)$$
for all $ u \in \ran(\Pi) \intersect \dom(\Pi)$.
\end{enumerate}

The authors of \cite{AKMC}
reveal that (H1)-(H3)
are adequate to set up the necessary operator theoretic framework.
However, as we have 
noted before, the full set of assumptions (H1)-(H8)
are necessary to obtain the desired estimates.
It is under these assumptions that we present
the main theorem of this paper.
\begin{theorem}
\label{Thm:HypRes:Main}
Let $\Spa$, $(\Gamma,B_1,B_2)$
satisfy (H1)-(H8). Then,
$\Pi_B$ satisfies the quadratic
estimate
$$
\int_{0}^\infty \norm{Q_t^B u}^2\ \frac{dt}{t}
	\simeq \norm{u}^2$$ 
for all $u \in \close{\ran(\Pi_B)} \subset \Lp{2}(\Spa,\C^N)$
and hence has a bounded $H^\infty$ functional 
calculus.
\end{theorem}

Let $E_B^{\pm} = \chi^{\pm}(\Pi_B)$, where
$\chi^{+}(\zeta) = 1$ when $\re(\zeta) > 0$
and $0$ otherwise, and similarly,
$\chi^{-}(\zeta) = 1$ when $\re(\zeta) < 0$
and $0$ otherwise.
We have the following
corollary resembling Corollary 2.11 in \cite{AKMC}.

\begin{corollary}[Kato square root type estimate]
{ \hfill }
\begin{enumerate}[(i)]
\item There is a spectral decomposition
	$$\Lp{2}(\Spa,\C^N) = \nul(\Pi_B) \oplus E_B^{+} \oplus E_B^{-}$$
	(where the sum is in general non-orthogonal), and
\item 	$\dom(\Gamma) \intersect \dom(\adj{\Gamma}_B) = \dom(\Pi_B) = \dom(\sqrt{\Pi_B^2})$
	with
	$$\norm{\Gamma u} + \norm{\Gamma_B u} 
		\simeq \norm{\Pi_B u} 	
		\simeq \norm{\sqrt{\Pi_B^2}u}$$
	for all $u \in \dom(\Pi_B)$.
\end{enumerate} 
\end{corollary} 
\section{Abstract dyadic decomposition}
\label{Dya}

We begin this section by
quoting Theorem 11 in \cite{Christ}.

\begin{theorem}
\label{Thm:Dya:Christ}
There exists a countable collection of open subsets
$$\set{\Q[\alpha]^k \subset \Spa: k \in \In, \alpha \in I_k}$$
with each $z_\alpha^k \in \Q[\alpha]^k$, 
where $I_k$ are index sets (possibly finite),
and constants $\delta \in (0,1)$, 
$a_0 > 0$, $\eta > 0$ and 
$C_1, C_2 < \infty$ satisfying:
\begin{enumerate}[(i)]
\item For all $k \in \In$,
	$\mu(\Spa \setminus \union_{\alpha} \Q[\alpha]^k) = 0$,
\item If $l \geq k$, either $\Q[\beta]^l \subset \Q[\alpha]^k$
	or $\Q[\beta]^l \intersect \Q[\alpha]^k = \emptyset$,
\item For each $(k,\alpha)$ and each $l < k$
	there exists a unique $\beta$ such that
	$\Q[\alpha]^k \subset \Q[\beta]^l$,
\item $\diam \Q[\alpha]^k \leq C_1 \delta^k$,
\item $ B(z_\alpha^k, a_0 \delta^k) \subset \Q[\alpha]^k$, 
\item For all $k, \alpha$ and for all $t > 0$, 
	$\mu\set{x \in \Q[\alpha]^k: d(x, \Spa\setminus\Q[\alpha]^k) \leq t \delta^k} \leq C_2 t^\eta \mu(\Q[\alpha]^k).$
\end{enumerate} 
\end{theorem}

Define $\DyQ^k = \set{\Q[\alpha]^k: \alpha \in I_k}$
to be the \emph{level} $k$ dyadic cubes
and $\DyQ = \union_{k} \DyQ^k$
to be the collection of dyadic cubes.
For $\Q[\alpha]^k \in \DyQ^k$, define
the \emph{length} as
$\len(\Q[\alpha]^k) = \delta^k$ and
the \emph{centre} as $z_\alpha^k$.

It is easy to see that each $\DyQ^k$ is
a mutually disjoint collection. Furthermore,
we have $\bnd (\union \DyQ^k) = \union_{Q \in \DyQ^k} \bnd Q$.
These facts coupled with the assumption $\mu(B(x,r)) > 0$ implies that
$\Spa = \close{\union \DyQ^k}$.

Fix a cube $\Q \in \DyQ^j$ and
denote the centre of this
cube by $z$. We are interested in 
counting the number of cubes
inside ``shells'' centred
from this cube.
We begin with the following
definition.
\begin{definition}
Whenever $k \geq 1$, define
$$\sC_{k} = \set{\Q[\alpha]^j \in \DyQ^{j}: (k-1) C_1\delta^j \leq d(z,z_{\alpha}^j) \leq kC_1\delta^j}.$$
Also, let $\tilde{\sC}_k = \set{\Q[\alpha]^j \in \DyQ^j: d(z, z_\alpha^j) \leq kC_1 \delta^j}$.
\end{definition}

It is easy to see that $\DyQ^j = \union_{k\geq1} \sC_k$.
We compute a bound for $\card \sC_k$
(where $\card S$ denotes the \emph{cardinality}
of a set $S$).
First, we have the following proposition
describing the distance of points in $\union \sC_k$
to $z$.
\begin{proposition}
\label{Prop:Dya:Est}
Let $\Q[\alpha]^j \in \sC_k$. Then,
\begin{enumerate}[(i)]
\item $0 \leq d(z,x) \leq (k+1)C_1\delta^j$ for all $x \in \Q[\alpha]^j$ when $k \leq 2$, and 
\item $\frac{1}{3}k C_1\delta^j \leq d(z,x) \leq (k+1)C_1\delta^j$ for all $x \in \Q[\alpha]^j$ when $k \geq 3$.
\end{enumerate}
\end{proposition}
\begin{proof}
Fix $\Q[\alpha]^j \in \sC_k$
and fix $x \in \Q[\alpha]^j$. Then,
$$
d(x,z) \leq d(x,z_\alpha^j)  + d(z_\alpha^j, z)
	\leq \diam \Q[\alpha]^j + kC_1\delta^j
	\leq (k+1)C_1\delta^j.$$
Also,
$$(k-1)C_1\delta^j \leq d(z,z_\alpha^j) \leq d(x,z) + d(x,z_\alpha^j) \leq d(x,z) + C_1\delta^j.$$
Combining these two estimates
we have
$$ (k-2)C_1\delta^j \leq d(z,x) \leq (k+1)C_1\delta^j.$$
This gives us (i). To obtain (ii), note that
whenever $k \geq 3$ we have $\frac{1}{3}k \leq k - 2$.
\end{proof}

Next, we   
compare two balls
which are separated by an arbitrary distance.
In the following proposition 
(and indeed the rest of the paper), 
let us fix $p = \log_2(C_D)$, where
$C_D$ is the doubling constant.
\begin{proposition}
\label{Prop:Dya:BallCompare}
Fix balls $B(x,r), B(y,r) \subset \Spa$. Then,
for all $\epsilon > 0$, 
\begin{align*}
2^{-p} \cbrac{\frac{d(x,y) + r + \epsilon}{r}}^{-p} \mu (B(y,r)) &\leq \mu (B(x,r))\\
&\leq 2^p \cbrac{\frac{d(x,y) + r + \epsilon}{r}}^p \mu (B(y,r)).
\end{align*}
\end{proposition}
\begin{proof}
Fix $\epsilon > 0$ and note that 
$$B(x,r), B(y,r) \subset B(x, d(x,y) + r + \epsilon), B(y, d(x,y) + r + \epsilon).$$
Therefore, 
\begin{align*}
\mu (B(y,r))
	&\leq \mu \cbrac{B\cbrac{x, \frac{d(x,y) + r + \epsilon}{r} r}}\\
	&\leq 2^p \cbrac{\frac{d(x,y) + r + \epsilon}{r}}^p \mu (B(x,r)).
	\end{align*}
Similarly, we have
\begin{align*}
\mu (B(x,r))
	&\leq \mu \cbrac{B\cbrac{y, \frac{d(x,y) + r + \epsilon}{r} r}}\\
	&\leq 2^p \cbrac{\frac{d(x,y) + r + \epsilon}{r}}^p \mu (B(y,r))
\end{align*}
which establishes the claim.
\end{proof}

We make a parenthetical remark
that our assumption $0 < \mu(B(x,r)) < \infty$ for all $x\in \Spa$
and $r > 0$ is not strong
since by the previous proposition, 
coupled with the doubling property,
allow us to recover this assumption
if we only required $0 < \mu(B(x_0,r_0)) < \infty$
to hold for some $x_0 \in \Spa$ and $r_0 > 0$.

We now return back to the problem
of estimating $\card \sC_k$.
The reader will observe that we
have been generous in our calculations.
\begin{proposition}
\label{Prop:Dya:CCount} 
We have $\card \tilde{\sC}_k \leq C k^{2p}$
where
$$C = 4^p  \cbrac{\frac{C_1 + 2 a_0}{a_0}}^p\cbrac{\frac{2C_1}{a_0}}^p.$$
In particular, $\card \sC_k \leq Ck^{2p}$.
\end{proposition}
\begin{proof}
Fix $k \geq 1$.
Set $\epsilon = r = a_0 \delta^j$ and then
$$
d(z,z_\alpha^j) + r + \epsilon 
	\leq k C_1 \delta^j + 2a_0 \delta^j
	\leq (C_1 + 2a_0) \delta^j k$$
when $\Q[\alpha]^j \in \tilde{\sC}_k$.
By Proposition \ref{Prop:Dya:BallCompare},
$$
2^{-p} \cbrac{\frac{C_1 + 2a_0}{a_0}}^{-p} k^{-p} \mu (B(z, a_0 \delta^j))
	\leq \mu (B(z_\alpha^k, a_0\delta^j)).$$

Now, note that by Proposition \ref{Prop:Dya:Est}, 
we have
$\sup_{x \in \Q[\alpha]^j} d(x,z) \leq (k+1)C_1\delta^j$
and so $\union\tilde{\sC}_k \subset B(z,(k+1)C_1\delta^j)$.
Then,
\begin{align*}
\mu (B(z, (k+1)C_1\delta^j ))
	&\leq 2^p \cbrac{\frac{(k+1)C_1}{a_0}}^p \mu (B(z,a_0\delta^j))\\
	&\leq 2^p \cbrac{\frac{2C_1}{a_0}}^p k^p \mu (B(z,a_0\delta^j)).
\end{align*}

Since 
$\mu (B(z,a_0\delta^j)) < \infty$
and by combining
the two estimates, and the
fact that $B(z_\alpha^k, a_0\delta^j) \subset \Q[\alpha]^j$
for each $\Q[\alpha]^j \in \tilde{\sC}_k$, 
we compute
\begin{align*}
\card \sC_k 
	&\leq 2^p \cbrac{\frac{2C_1}{a_0}}^p k^p\ 2^p \cbrac{\frac{C_1 + 2 a_0}{a_0}}^p k^p\\
	&= 4^p  \cbrac{\frac{C_1 + 2 a_0}{a_0}}^p\cbrac{\frac{2C_1}{a_0}}^p k^{2p}.\end{align*}
The observation that $\sC_k \subset \tilde{\sC}_k$ 
completes the proof.
\end{proof}

We have the following important
consequences. They are
useful in many of the calculations 
in \S\ref{HarmAnal}.
Following the notation in \cite{AKMC},
we write $\inprod{x}= 1 + \modulus{x}$.

\begin{corollary}
\label{Cor:Dya:Count}
Fix $\delta^{j+1} < t \leq \delta^j$
and a cube $\Q \in \DyQ^j$. Then,
$$
\sum_{R \in \DyQ^{j}} \inprod{\frac{\dist(R,\Q)}{t}}^{-M}
 \leq C \cbrac{ 1 + 4^p + \cbrac{\frac{3}{C_1}}^M\  \sum_{k=3}^\infty k^{2p - M}}$$
with $C$ being the constant in the previous proposition.
\end{corollary}
\begin{proof}
First, we note that
$$ 1 \leq 1 + \frac{\dist(R,Q)}{t} \quad\text{and}\quad \frac{\dist(R,Q)}{\delta^j} \leq 1 + \frac{\dist(R,Q)}{t}.$$
Then,
\begin{align*}
\sum_{R \in \DyQ^{j}} \inprod{\frac{\dist(R,\Q)}{t}}^{-M}
	&\leq  \card \sC_1 + \card \sC_2 + \sum_{k=3}^\infty \sum_{R\in \sC_k} \cbrac{\frac{\delta^j}{d(R,\Q)}}^M  \\
	&\leq   C + C2^{2p} +  \sum_{k=3}^\infty \card \sC_k \cbrac{\frac{\delta^j}{\frac{1}{3}k C_1 \delta^j}}^M \\
	&\leq C \cbrac{ 1 + 4^p + \cbrac{\frac{3}{C_1}}^M\  \sum_{k=3}^\infty k^{2p - M}}.
\end{align*}
\end{proof}

\begin{corollary}
\label{Cor:Dya:Conv}
For each $M > 2p + 1$, there exists
a constant $A_M > 0$ such that
$$\sup_{Q} \sum_{R \in \DyQ^{j}} 
\inprod{\frac{\dist(R,\Q)}{t}}^{-M} \leq A_M.$$
\end{corollary}
\section{Maximal functions and Carleson Measures}
\label{MaxCarl}

A full treatment of the classical theory
of maximal functions and Carleson measures
can be found in
\S4 of \cite{stein:harm} by Stein. 
The objects of interest that
we define in this section
are taken from this book
\emph{mutatis mutandis}.
Furthermore, we 
refer the reader to 
\cite{Heinonen}
by Heinonen
and \cite{CW} by Coifman and Weiss 
as two excellent expositions
that touch on some of the issues
and ideas presented here.

For a measurable subset $S$
with $0 < \mu(S) < \infty$
and $f \in \Lp[loc]{1}(\Spa, \C^N)$, 
we define 
the \emph{average} of $f$
on $S$ by $\fint_{S} f = \mu(S)^{-1} \int_{S} f$.
Then, we make the following definition.
\begin{definition}[Maximal function]
\index{Maximal function!Uncentred}
\nomenclature{$\Max f$}{Uncentred maximal function of $f$.}
Let $f \in \Lp[loc]{1}(\Spa,\C^N)$.  Define 
the uncentred maximal function of $f$ by:
$$ \Max f(x) = \sup_{B \ni x} \fint_{B} \modulus{f}\ d\mu$$
where the supremum is taken over all balls $B$ containing $x$.
\end{definition}

We want to deduce that this $\Max$
exhibits a weak type $(1,1)$ estimate
and is bounded in $\Lp{p}(\Spa,\C^N)$ for $p > 1$.
The proof of the following theorem is standard
via the Vitali type covering Theorem 1.2 in \cite{CW}.
 
\begin{theorem}[Maximal theorem]
\index{Maximal theorem}
\label{Thm:Max:Max}
There exists a constant $C_1 > 0$ such that 
whenever $f \in \Lp{1}(\Spa,\C^N)$, we have
$$
\mu(\set{x \in \Spa: \Max f(x) > \alpha}) \leq \frac{C_1}{\alpha} \int_{\Spa}\modulus{f}\ d\mu.$$
Whenever $f \in \Lp{q}(\Spa,\C^N)$ with $q > 1$,
$$
\norm{\Max f}_q \leq C_q \norm{f}_q$$
where $C_q > 0$ is a constant.
\end{theorem}

In order to set up a theory of Carleson measures, 
we require an \emph{upper half space}.
We define this to be $\Spa_+ = \Spa \times \R^+$
where $\R^+ = (0,\infty)$.
The \emph{cone} over a point $x \in \Spa$
is then defined as
$\Cone(x) = \set{(y,t) \in \Spa_+: d(x,y) < t}$
and this leads to the following.
\begin{definition}[Nontangential maximal function]
\index{Maximal function!Non-tangential}
\nomenclature{$\Max^\ast f(x)$}{Non-tangential maximal function of $f$ at point $x$.}
Let $f \in \Lp[loc]{1}(\Spa_+, \C^N)$. Define
$$
\Max^\ast f(x) = \sup_{(y,t) \in \Cone(x)} \modulus{f(y,t)}.$$ 
\end{definition}

Like its classical counterpart,
this maximal function is measurable.
This is the content of the following
proposition.
\begin{proposition}
The set $\set{x \in \Spa: \Max^\ast f(x) > \alpha}$
is open and hence $\Max^\ast f$ is measurable.
\end{proposition}
\begin{proof}
Fix $x \in \Spa$
with $\Max^\ast f(x) >\alpha$. 
Then, there exists a $(y,t) \in \Cone(x)$ such that
$\modulus{f(y,t)} > \alpha$.
Consider the ball $B(y,t)$ and take any $z \in B(y,t)$.
Note that since $d(z,y) < t$ we have 
$(y,t) \in \Cone(z)$ and so 
$\Max^\ast f(z) > \alpha$.
Therefore,  
$x \in B(y,t) \subset \set{x \in \Spa: \Max^\ast f(x) > \alpha}$.
\end{proof}

Therefore, we define the following
function space in an analogous
way to the classical theory.
\begin{definition}[Nontangential function space]
\index{Non-tangential function space}
\nomenclature{$\Nfs$}{Space of non-tangential functions.}
Let $\Nfs$ denote the space of 
Borel measurable functions $f:\Spa_+ \to \C$
such that $\Max^\ast f \in \Lp{1}(\Spa)$. We
equip this space with the norm
$\norm{f}_{\Nfs} = \norm{\Max^\ast f}_1$.
\end{definition}

Now, let $B = B(x,r)$ 
and define the \emph{tent over $B$}
as 
$$\Tent(B) = \set{ (y,t) \in \Spa_+: d(x,y) \leq r - t}.$$
For an arbitrary open set $O \subset \Spa$, 
we define the \emph{tent over $O$} by 
$\Tent(O) = 
\Spa_+ \setminus \union_{x \in \Spa \setminus O} \Cone(x)$.
The following is an equivalent
characterisation of $\Tent(O)$.
\begin{proposition}
\label{Prop:Max:TentForm}
Whenever $(x,t) \in \Tent(O)$
we have that
$$(x,t) \in \Tent(B(x,d(x,\Spa \setminus O)))$$ 
and in particular,
$\Tent(O) = \union_{x \in O} \Tent(B(x,d(x,\Spa \setminus O))).$
\end{proposition}
\begin{proof}
First, note that by de Morgen's law, we can conclude that
$\Tent(O) = \intersect_{y \in \Spa \setminus O} \Spa_+ \setminus \Gamma(y)$.
Fix $(x,t) \in \Tent(O)$. So,
$(x,t) \in \Spa_+ \setminus \Gamma(y)$ for all $y \in \Spa \setminus O$.
That is, for all $y \not\in O$, 
we have $(x,t) \not \in \Gamma(y)$ which implies
$d(x,y) \geq t$. Therefore, $d(x,\Spa \setminus O) \geq t$.
Then, by the definition of $\Tent(B(x,r))$
and setting $r = d(x,\Spa \setminus O)$,
we conclude $(x,t) \in \Tent(B(x,d(x,\Spa\setminus O)))$.
The converse inclusion is easy 
since $B(x,d(x,\Spa\setminus O)) \subset O$.
\end{proof}

\begin{definition}[Carleson function]
\index{Carleson!Function}
\nomenclature{$\Cf(\nu)$}{Carleson function of Borel measure $\nu$.}
Let $\nu$ be any Borel measure on $\Spa_+$.
Define
$$\Cf(\nu)(x) = \sup_{B \ni x} \frac{\nu(\Tent(B))}{\mu(B)}.$$
\end{definition}

\begin{definition}[Space of Carleson measures]
\index{Carleson!Space of measures} 
\nomenclature{$\Carl$}{Space of Carleson measures.}
We define $\Carl$ to be the space of measures
$\nu$ that are Borel on $\Spa_+$ and such that
$\Cf(\nu)$ is bounded. Such a measure is
called a \emph{Carleson measure} and we define
$$\norm{\nu}_{\Carl}  = \sup_{x \in \Spa} \Cf(\nu)(x)$$
to be the \emph{Carleson norm}.
\end{definition}

Since we have a dyadic structure,
we define the 
\emph{Carleson box}
over $\Q \in \DyQ$
by 
$\CBox_{\Q} = \close{\Q} \times (0, \len(\Q)]$.
Unlike the classical 
definition, we
are forced to take $\close{\Q}$
since $\DyQ$ is only 
guaranteed to cover $\Spa$
almost everywhere.
The importance
of this subtlety 
will become apparent in the 
proof of the following
proposition that
provides an alternative characterisation
of a Carleson measure.

\begin{proposition}
Let $\nu$ be a Borel measure on $\Spa_+$.
Then the statement 
$$\sup_{B} \frac{\nu(\Tent(B))}{\mu(B)}  < \infty
\qquad\text{for every ball $B$}$$
is equivalent to the statement 
$$\sup_{\Q} \frac{\nu(\CBox_{\Q})}{\mu(\Q)}  < \infty
\qquad\text{for every $\Q \in \DyQ$}.$$
\end{proposition}
\begin{proof}
First, fix $\Q \in \DyQ^j$ and let $x_{\Q}$ be its centre.
Then, we have that $\Q \subset B(x_{\Q}, C_1\delta^j)$. Then,
certainly, $\CBox_{\Q} \subset \Tent(B(x_{\Q}, (C_1 + 2) \delta^j))$.
So,
\begin{multline*}
\nu(\CBox_{\Q}) \leq \nu(\Tent(B(x_{\Q}, (C_1 + 2)\delta^j)))
	\leq \norm{\nu}_\Carl \mu(B(x_{\Q}, (C_1 + 2)\delta^j) \\
	\leq 2^\p \cbrac{\frac{C_1 + 2}{a_0}}^\p \norm{\nu}_\Carl  \mu(B(x_{\Q},a_0 \delta^j))
	\leq 2^\p \cbrac{\frac{C_1 + 2}{a_0}}^\p \norm{\nu}_\Carl \mu(\Q).
\end{multline*}

The converse is harder.
Fix $B = B(x,r)$ and let $j \in \In$ such that
$\delta^{j+1} < r \leq \delta^j$.
Let $N(B) = \set{\Q \in \DyQ^j: \Q \intersect B \neq \emptyset}$.
It is an easy fact that $N(B) \neq \emptyset$.
\begin{enumerate}[(i)]
\item 
First, we claim that $B \subset \union_{\Q \in N(B)} \close{\Q}$.
Suppose $y \in B$ but $y \not \in \union N(B)$. That is,
$y \not \in \Q$ for all $\Q \in \DyQ^j$. 
Thus, 
there exists a $\Q \in \DyQ^j$ such that
$y \in \bnd\Q$. That is, for every $\epsilon > 0$,
$B(y,\epsilon) \intersect \Q \neq \emptyset$.
But there exists an $\epsilon > 0$ such that
$B(y,\epsilon) \subset B$, and so 
$\Q \intersect B \neq \emptyset$. This
means that $\Q \in N(B)$ and establishes
the claim.

\item
Fix $\Q \in N(B)$ as a reference cube and let $\Q' \in N(B)$ be
any other cube. Since $r < \delta^j$, we note that
$d(x,x_{\Q}), d(x,x_{{\Q}'}) \leq \delta^j + C_1 \delta^j$. Therefore,
$d(x_{\Q}, x_{{\Q}'}) \leq 2(C_1 + 1) \delta^j$. That is, 
all the centres of cubes $\Q' \in N(B)$ are 
inside the ball $B(x_{\Q}, 2(C_1 + 1) \delta^j)$
and hence $\tilde{\sC}_{2(C_1 + 1)}$. Thus,
by Proposition \ref{Prop:Dya:CCount},
$$\card N(B) \leq \card \tilde{\sC}_{2(C_1 +1)}
	\leq C 2^p (C_1 + 1)^{2p}.$$

\item
Now, suppose that $(y,t) \in \Tent(B)$.
That is, $y \in B$ and 
We have $d(y,t) \leq r - t \leq \delta^j$.
By (i), there exists a cube $\Q \in N(B)$
such that $y \in \close{\Q}$.
Therefore, $(y,t) \in \CBox_{\Q} = \close{\Q}$
and shows that
$\Tent(B) \subset \union_{\Q \in N(B)} \CBox_{\Q}$.

\item
Fix $\Q \in N(B)$ and so $d(x, x_{\Q}) \leq (C_1 + 1) \delta^j$.
Set $\epsilon = r = \delta^{j+1}$ in Proposition \ref{Prop:Dya:BallCompare}
so that
\begin{align*}
\mu( B(x_{\Q}, \delta^{j+1})) 
	&\leq 2^p \cbrac{ \frac{(C_1+ 1)\delta^{j} + 2\delta^{j+1}}{\delta^{j+1}}} 
		\mu (B(x,\delta^{j+1})) \\
	&\leq 2^p ((C_1 +1)\delta^{-1} + 2)^p \mu(B(x,r)).
\end{align*}
\end{enumerate}
Now, by combining (i) - (iv),
\begin{multline*}
\nu(\Tent(B)) 
	\leq \sum_{\Q \in N(B)} \nu(\CBox_{\Q})
	\lesssim \sum_{\Q \in N(B)} \mu (B(x_{\Q}, C_1 \delta^j)) \\ 
	\lesssim \sum_{\Q \in N(B)} \mu(B(x_{\Q}, \delta^{j+1})) 
	\lesssim \card N(B) \mu(B(x,r)) 
	\lesssim \mu(B(x,r))
\end{multline*}
which completes the proof.
\end{proof}

We quote the following covering
theorem of Whitney given as Theorem 1.3 in \cite{CW}.

\begin{theorem}[Whitney Covering Theorem]
\label{Thm:Max:Whitney}
Let $O \subsetneqq \Spa$ be open. Then,
there exists a set of balls $\sE = \set{B_j}_{j \in \Na}$
and a constant $c_1 < \infty$ independent of $O$ 
such that
\begin{enumerate}[(i)]
\item The balls in $\sE$ are mutually disjoint, 
\item $O = \Union_{j \in \Na} c_1 B_j$,
\item $4c_1 B_j \not\subset O$.
\end{enumerate}
\end{theorem}

This allows us to prove the following theorem of Carleson.

\begin{theorem}[Carleson's Theorem]
\label{Thm:Max:Carl} 
Let $f \in \Nfs$ and $\nu \in \Carl$. Then,
$$
\iint_{\Spa_{+}} \modulus{f(x,t)}\ d\nu(x,t)
	\lesssim \norm{f}_{\Nfs} \norm{\nu}_{\Carl}$$
where the constant depends only 
on $\p$ and the Whitney constant $c_1$.
\end{theorem}
\begin{proof}
\begin{enumerate}[(i)]
\item
We prove
$\set{(x,t) \in \Spa_+: \modulus{f(x,t)} > \alpha} \subset \Tent(E_\alpha)$
where $E_{\alpha} = \set{x \in \Spa: \Max^\ast f(x) > \alpha}$.
Fix $(x,t) \in \Spa_+$ such that $\modulus{f(x,t)} > \alpha$. Then,
whenever $y \in B(x,t)$, we also have $x \in B(y,t)$ and
$$\Max^\ast f(y) = \sup_{t > 0} \sup_{z \in B(y,t)} \modulus{f(z,t)}
	> \modulus{f(x,t)} > \alpha.$$
Therefore, $B(x,t) \subset E_\alpha$ and
$(x,t) \in \Tent(B(x,t)) \subset \Tent(E_\alpha)$.

\item 
Let $O \subsetneqq \Spa$ be an open set, and
let $\sE = \set{B_j}_{j \in \Na}$ 
be the Whitney covering
guaranteed by Theorem \ref{Thm:Max:Whitney}. We prove that
$$ \Tent(O) \subset \union_{j} \Tent(9c_1 B_j).$$

Fix $x \in O$ and let 
$(x,t) \in T(B(x,d(x,\Spa\setminus O)))$. 
Then, there exists a ball 
$B_j = B_j(x_j,r_j) \in \sE$ such that
$x \in c_1 B_j$. Let $y \in B(x,d(x,\Spa\setminus O))$.
Since $4c_1 B_j \intersect \Spa \setminus O$,
for any $z \in \Spa\setminus O$
$d(y, \Spa \setminus O) \leq d (x,z) \leq 8c_1 r_j$
Then, 
$$
d(y,x_j) \leq d(y, x) + d(x, x_k)
	\leq d(x,\Spa \setminus O) + d(x,x_k)
	< 8c_1 r_j + c_1 r_j  = 9c_1 r_j.$$
This proves that
$B(x, d(x,\Spa \setminus O)) \subset 9c_1B_j$
and so
$\Tent(B(x, d(x,\Spa \setminus O))) \subset \Tent(9c_1B_j).$
We apply Proposition \ref{Prop:Max:TentForm} to 
conclude that
$ \Tent(O) \subset \union_{j} \Tent(9c_1 B_j).$

\item
Now, we prove that there exists a constant
$C > 0$ such that for all open sets $O \subset \Spa$,
$$ \nu(\Tent(O)) \leq C \norm{\nu}_{\Carl} \mu(O).$$

First assume that $O = \Spa$. 
If $\mu(\Spa) = \infty$, then there
is nothing to prove. So suppose otherwise.
Now, for any $x \in \Spa$
and any ball $B_r = B(x,r)$, 
$$
\frac{1}{\mu(B_r)} \nu(\Tent(B_r)) \leq \Cf(\nu)(x) \leq \norm{\nu}_{\Carl}$$
and therefore,
$\nu(\Tent(B_r)) \leq \norm{\nu}_{\Carl} \mu(\Spa)$
for every ball $B_r$ of radius $r$.
Now, $\ind{{\Tent(B_{n})}} \leq 1$ for each $n \in \Na$
and $\ind{{\Tent(B_{n})}} \to \ind{\Tent(\Spa)}$ and $n \to \infty$ 
pointwise. Then, by application of Dominated Convergence Theorem, 
$$
\nu(\Tent(\Spa)) 
	= \int_{\Spa_+} \lim_{n \to \infty} \ind{{\Tent(B_n)}}\ d\nu
	= \lim_{n \to \infty}\int_{\Spa_+} \ind{{\Tent(B_n)}}\ d\nu
	\leq \norm{\nu}_{\Carl} \mu(\Spa).$$ 

Now, consider the case when $O \subsetneqq \Spa$. 
Then, by (ii) and the subadditivity of the measure,
\begin{multline*}
\nu(\Tent(O))
	\leq \sum_{j} \nu(\Tent(9c_1B_j))
	\leq \norm{\nu}_{\Carl} \sum_{j} \mu(9c_1 B_j) \\
	\leq 2^\p (9c_1)^\p \norm{\nu}_{\Carl} \sum_{j} \mu(B_j)
	\leq (18c_1)^\p \norm{\nu}_{\Carl} \mu(O).
\end{multline*}

\item
By (i) and (iii),
$$
\nu\set{(x,t) \in \Spa_+: \modulus{f(x,t)} > \alpha}
	\lesssim \norm{\nu}_{\Carl} \mu \set{x \in \Spa: \Max^\ast f(x) > \alpha}$$
and integrating both sides with respect to $\alpha$ completes the proof.
\end{enumerate}
\end{proof}
\section{Harmonic Analysis of $\Pi_B$}
\label{HarmAnal}

Let $\DyQ_t = \DyQ^j$ for
$\delta^{j+1} < t \leq \delta^j$.
Following the structure of the 
proof in \cite{AKMC},
for $t \in \R^+$,
we define the \emph{dyadic averaging operator}
$\Av_t:\Hil \to \Hil$ as
$$
\Av_t(x) = \sum_{\Q \in \DyQ_t} \ind{\Q}(x) \fint_{\Q} u\ d\mu$$
when $x \in \union\DyQ_t$ and $0$
elsewhere.
A straightforward calculation shows that
$\Av_t \in \bddlf(\Hil)$
and $\norm{\Av_t} \leq 1$ uniformly in $t$.
Then, the \emph{principal part}
is defined as $\pri_t(x)w = (\Theta_t^B \omega)(x)$
for $w \in \C^N$ 
and where $\omega(x) = w$ for all $x \in \Spa$.

Following \cite{AKMC}, 
to prove Theorem \ref{Thm:HypRes:Main}
as a consequence
of Proposition \ref{Prop:HypRes:Main},
we need to show that 
$$\int_{0}^\infty \norm{\Theta_t^B P_t u}^2\ \frac{dt}{t} \lesssim \norm{u}^2$$
for $u \in \ran(\Pi)$.
Thus, we follow the 
paradigm in \cite{AKMC}, \cite{AKM2} and
\cite{Morris} and decompose this problem in the following way:
\begin{multline*}
\int_{0}^\infty \norm{\Theta_t^B P_t u}^2\ \frac{dt}{t}
	\leq \int_0^\infty \norm{\Theta_t^B P_t u - \pri_t \Av_t u}^2 \frac{dt}{t} \\
	+ \int_{0}^\infty \norm{\pri_t \Av_t (P_t - I)u}^2\ \frac{dt}{t}
	+ \iint_{\Spa_+} \modulus{\Av_t u(x)}^2 \modulus{\pri_t(x)}^2\ \frac{d\mu(x)dt}{t}.
\end{multline*}
The purpose of the first two
terms is to reduce the estimate down to the third term
which can be dealt
with a Carleson measure estimate.

\subsection{Off-Diagonal Estimates}

The following lemma is a primary tool in our
argument. Certainly, it was known to 
the authors of \cite{AKMC}
since they use a similar result
in the proof of 
their Proposition 5.2. The key difference
is that we use $\Lip \xi$
instead of $\norm{\grad \xi}_\infty$
to control the ``slope'' of our cutoff.
Furthermore, this lemma is
used later in our work
to construct Lipschitz substitutions where
\cite{AKMC}, \cite{AKM2} and \cite{Morris}
use smooth cutoff functions.
We include a detailed proof
of this lemma since it is central to
our work.

\begin{lemma}[Lipschitz separation lemma]
\label{Lem:Harm:LipSep}
Let $(X,d)$ be a metric space and 
suppose $E, F \subset X$ satisfy
$d(E,F) > 0$. Then, there exists
a Lipschitz function $\eta:X \to [0,1]$, 
and a set 
$\tilde{E} \supset E$
with $d(\tilde{E},F) > 0$
such that
$$\eta|_E = 1,\quad \eta|_{X \setminus \tilde{E}} = 0
\quad\text{and}\quad \Lip \eta \leq 4/d(E,F).$$ 
\end{lemma}
\begin{proof}
Define $\tilde{E} = \set{x \in X: d(x,E) < 1/4 d(E,F)}$. 
By construction, $E \subset \tilde{E}$ and from the triangle
inequality for $d$ and taking infima,
$$d(\tilde{E},F) + \sup_{x \in \tilde{E}} d(x,E) \geq d(E,F),$$
and since $\sup_{x\in\tilde{E}} d(x,E) \leq\dfrac{1}{4} d(E,F)$,
it follows that $d(\tilde{E},F) \geq\dfrac{3}{4} d(E,F) > 0$.

Now, define:
$$\eta(x) = \begin{cases} 1 - \frac{4 d(x,E)}{d(E,F)} &x \in \tilde{E}\\
			0 &x \not\in \tilde{E}
	\end{cases}.$$
We consider the three possible cases.
\begin{enumerate}[(i)]
\item 
First, suppose that $x,y \not\in\tilde{E}$.
Then,
$$ 
|\eta(x) - \eta(y)| = 0 \leq \frac{4d(x,y)}{d(E,F)}.$$

\item
Now, suppose that $x,y \in \tilde{E}$. By the triangle inequality,
we have
$d(x,z) \leq d(x,y) + d(y,z)$ and by taking 
an infima over $z \in E$ and 
invoking the symmetry of distance,
$\modulus{d(x,E) - d(y,E)} \leq d(x,y)$. Therefore,
\begin{multline*} 
\modulus{\eta(x) - \eta(y)}
	= \modulus{\frac{1 - 4d(x,E)}{d(E,F)} - 1 + \frac{4d(y,E)}{d(E,F)}} \\
	= \frac{4}{d(E,F)} \modulus{d(x,E) - d(y,E)}
	\leq \frac{4}{d(E,F)} d(x,y).
\end{multline*}

\item
Lastly, suppose that $x \in \tilde{E}$ and $y \not\in \tilde{E}$. 
Then $\eta(y) = 0$ and since $d(x,E) \leq \frac{1}{4}d(E,F)$,
$$|\eta(x) -\eta(y)| = |\eta(x)| = \eta(x) 
	= 1 - \frac{4 d(x,E)}{d(E,F)}
	= \frac{d(E,F) - 4d(x,E)}{d(E,F)}.$$
But we also have the triangle inequality
$d(E,x) + d(x,y) \geq d(y,E)$
and by the choice of $y$
we have that $d(y,E) \geq 1/4 d(E,F)$.
Therefore,
$d(x,y) \geq d(y,E) - d(x,E)
	\geq \frac{1}{4} d(E,F) - d(x,E)$
which implies that
$$
\frac{4d(x,y)}{d(E,F)} \geq \frac{d(E,F) - d(x,E)}{d(E,F)} 
	= \modulus{\eta(x) - \eta(y)}.$$
\end{enumerate}
\end{proof}

A preliminary and immediate
consequence is the following
off-diagonal estimates
resembling 
those in \S5.1 in \cite{AKMC}.
\begin{proposition}[Off-diagonal estimates]
\label{Prop:Harm:OffDiag}
Let $U_t$ be either $R_t^B$ for $t \in \R$
or $P_t^B, Q_t^B, \Theta_t^B$ for $t > 0$.
Then, for each $M \in \Na$, there exists 
a constant $C_M > 0$ (that depends
only on $M$ and the constants in (H1)-(H6)) such that
$$\norm{U_tu}_{\Lp{2}(E)} \leq C_M \inprod{\frac{\dist(E,F)}{t}}^{-M} \norm{u}_{\Hil}$$
whenever $E,F \subset \Spa$ are Borel sets
and $u \in \Hil$ with $\spt u \subset F$.
\end{proposition}

We omit the proof since it is essentially the
same as that of Proposition 5.2 in \cite{AKMC}.
The following is an immediate consequence.

\begin{corollary}
\label{Cor:Harm:Off}
Let $\Q \in \DyQ_t$ and $0 < s \leq t$
with $U_s$ as specified in the proposition.
Then,
$$
\norm{U_s u}_{\Lp{2}(\Q)}
	\leq C_M \sum_{R \in \DyQ_t} \inprod{\frac{\dist(R,Q)}{s}}^{-M} \norm{u}_{\Lp{2}(R)}$$
whenever $u \in \Hil$.
\end{corollary}

In our setting, it is more convenient
to deal with the
following function space rather
than $\Lp[loc]{2}$ as used in \cite{AKMC}.

\begin{definition}
\nomenclature{$\Lp[\DyQ_t]{2}(\Spa)$}{Functions $f:\Spa \to \C^N$ that are in $\Lp{2}(\Q)$ for each $\Q \in \DyQ_t$.}
We define $\Lp[\DyQ_t]{2}(\Spa,\C^N)$ to be
the space of measurable functions $f:\Spa \to \C^N$
such that on each $\Q \in \DyQ_t$,
$$ \int_{\Q} \modulus{f}^2\ d\mu < \infty.$$
We equip this space with the seminorms
$\norm{\mdot}_{\Lp{2}(\Q)}$ indexed by 
$\DyQ_t$.
\end{definition} 

We have the following observations analogous
to those on  page 478 in \cite{AKMC}. It follows from
Propositions \ref{Prop:Dya:Est}, 
\ref{Prop:Dya:BallCompare}, \ref{Prop:Dya:CCount}
coupled with the off-diagonal 
estimates and by choosing $M > \frac{5p}{2} + 1$.
We remind the reader that
$p = \log_2(C_D)$ where $C_D$ is the doubling
constant.

\begin{corollary}
\label{Cor:Harm:ContExt}\
There exists a $C' > 0$ such that 
for all $t > 0$,
$U_t$ extends to a continuous map
$U_t:\Lp{\infty}(\Spa, \C^N) \to \Lp[\DyQ_t]{2}(\Spa,\C^N)$
with 
$$\norm{U_t u}_{\Lp{2}(\Q)}
	\leq C' \mu(\Q)^{\frac{1}{2}} \norm{u}_{\Lp{\infty}}.$$
\end{corollary}

\begin{corollary}
\label{Cor:Harm:AvUniBound}
We have $\pri_t \in \Lp[\DyQ_t]{2}(\Spa, \bddlf(\C^N))$
and for all $\Q \in \DyQ_t$ satisfy
$$
\fint_{\Q} \modulus{\pri_t(x)}^2_{\bddlf(\C^N)}\ d\mu(x) \leq C'^2
$$
In particular, $\norm{\pri_t \Av_t}_{\bddlf(\Hil)} \leq C'$ uniformly for all $t > 0$.
The constant $C'$ is the same as that of the previous corollary.
\end{corollary}

\subsection{Weighted \Poincare inequality and bounding the first term}

Controlling the first term in \cite{AKMC} relies primarily
on the weighted \Poincare inequality as given in 
Lemma 5.4 in \cite{AKMC}.
We pursue a similar strategy 
and begin by noting the following simple consequence of (H8).

\begin{lemma}[Dyadic Poincar\'e]
\label{Lem:Harm:DPoin}
Whenever $\Q \in \DyQ_t$ and $r \geq C_1\delta^{-1}$ we have
$$
\int_{B(x_{\Q},rt)} \modulus{u(x) - u_{\Q}}^2\ d\mu(x)
	\lesssim r^{ p + 2}
	 \int_{B(x_{\Q},crt)} \modulus{t \Pi u(x)}^2\ d\mu(x)$$
for all $u \in \ran(\Pi) \intersect \dom(\Pi)$.
\end{lemma}

This yields the following proposition
analogous to Lemma 5.4 in \cite{AKMC}.

\begin{proposition}[Weighted Poincar\'e]
\label{Prop:Harm:WPoin}
Whenever $\Q \in \DyQ_t$ and $M >  p + 1$,
we have
$$
\int_{\Spa} \modulus{u(x) - u_{\Q}}^2 \inprod{\frac{d(x,\Q)}{t}}^{-M}\ d\mu(x)
	\lesssim \int_{\Spa} \modulus{t\Pi u(x)}^2 \inprod{\frac{d(x,\Q)}{t}}^{ p - M}\ d\mu(x)$$
for all $u \in \ran(\Pi) \intersect \dom(\Pi)$, where the constant
depends on $M$. 
\end{proposition}
\begin{proof}
Observe that 
for $M > 1$, we have 
$$\inprod{\frac{d(x,\Q)}{t}}^{-M} 
	\leq \frac{2C_1}{\delta} \inprod{\frac{d(x,x_{\Q})}{t}}^{-M}.$$
By evaluating the integral
$$\int_{\Spa}\int_{\theta(x)}^\infty \modulus{u(x) - u_{\Q}}\ d\nu(r)\ d\mu(x),$$
where $d\nu(r) = M r^{-M-1}\ dr$, and invoking
Lemma \ref{Lem:Harm:DPoin} along with Fubini's Theorem
establishes the claim. 
\end{proof} 

This leads to the following proposition
which bounds the first term.

\begin{proposition}[First term inequality]
\label{Prop:Harm:FirstPrin}
Whenever $u \in \ran(\Pi)$, we have
$$
\int_0^\infty \norm{\Theta_t^B P_t u - \pri_t \Av_t P_t u}^2
	\lesssim \norm{u}^2.$$ 
\end{proposition}

We omit the proof
since it is very similar to the proof
of Proposition 5.5 in \cite{AKMC}. 
It is a simple matter of verification
using Corollary \ref{Cor:Dya:Conv} and
invoking the weighted \Poincare inequality.

\subsection{Bounding the second term}

The bounding of the second term relies on 
a suitable substitution for Lemma 5.6 in \cite{AKMC}.\
The crux of the argument is to be able to 
perform a cutoff ``close'' to the boundary of the dyadic
cube in question. First, we define
the following sets.
 
\begin{definition}[$\cE_\tau,\tilde{\cE}_\tau$]
Let $\Q \in \DyQ_t$ and $\tau \leq t$ Define
$$
\cE_\tau = \set{x \in \Q: d(x,\Spa\setminus\Q) > \frac{a_0\tau}{2}},\ \tilde{\cE}_\tau = \set{x \in \Q: d(x,\Spa\setminus\Q) \leq \frac{a_0\tau}{2}}.
$$
\end{definition}

The following proposition renders
a suitable Lipschitz substitution to the smooth
cutoff used in Lemma 5.6 in \cite{AKMC} 
and Lemma 5.7 in \cite{Morris}.

\begin{proposition}
\label{Prop:Harm:LipCube}
There exists a
Lipschitz function $\xi:\Q \to [0,1]$
such that $\xi = 1$ on $\cE_\tau$, 
$\spt (\Lipp \xi) \subset \tilde{\cE}_\tau$,
and 
$$\Lip \xi \leq \frac{16}{a_0\tau}.$$
\end{proposition}
\begin{proof}
Set 
$$F = \set{x \in \Q: d(x,\Spa\setminus\Q) \leq \frac{a_0\tau}{4}}$$
and note that $F \subset \tilde{\cE}_\tau$.
Then,
$$
\frac{a_0 \tau}{2} \leq \dist(\Spa\setminus\Q, \cE_\tau)
	\leq \dist(\cE_\tau, F) + \dist(\Spa\setminus\Q, F)
	\leq \dist(\cE_\tau, F) + \frac{a_0\tau}{4}$$
and so
$\dist(\cE_\tau, F) > \frac{a_0\tau}{4}$. 
By application of Lemma \ref{Lem:Harm:LipSep}, 
we find $\xi = 1$ on $\cE_\tau$, 
$\xi = 0$ on $\Q \setminus F$
and
$$\Lip \xi \leq \frac{4}{\frac{a_0 \tau}{4}} = \frac{16}{a_0\tau}.$$

Now, fix $x \in \cE_\tau$.
It is a simple matter to verify that $\cE_\tau$ is open and nonempty. So 
there exists an $\epsilon_0 > 0$ such that $B(x,\epsilon_0) \subset \cE_\tau$.
Therefore,
\begin{align*}
\Lipp \xi(x) &= \limsup_{y \to x} \frac{\modulus{\xi(x) - \xi(y)}}{d(x,y)}\\
	&= \lim_{\epsilon \to 0} 
	\sup\set{\frac{\modulus{\xi(x) - \xi(y)}}{d(x,y)}: y \in \cE_\tau \intersect B(x,\epsilon)\setminus\set{a}}
	= 0.
	\end{align*}
Thus, $\spt \xi \subset \tilde{\cE}_{\tau}$.
\end{proof}

This enables
us to prove the following lemma.
It is of key 
importance in bounding the second
term, as well as in the Carleson 
measure estimate
which allows us to bound the
last term.
 
\begin{lemma}
\label{Lem:Harm:UpsilonHom}
Let $\Upsilon$ be $\Gamma, \adj{\Gamma}$ or $\Pi$. Then,
whenever $\Q \in \DyQ_t$, 
$$
\modulus{\fint_{\Q}  \Upsilon u\ d\mu}^2
	\lesssim \frac{1}{t^\eta}
		\cbrac{\fint_{\Q} \modulus{u}^2\ d\mu}^{\frac{\eta}{2}}
		\cbrac{\fint_{\Q} \modulus{\Upsilon u}^2\ d\mu}^{1 - \frac{\eta}{2}}$$
where the constant depends only on
$C_1,\ C_2,\ a_0,\ \eta$ and $p$.
\end{lemma}
\begin{proof}
Let
$\tau = 
 \cbrac{\fint_{\Q} \modulus{u}^2\ d\mu }^{\frac{1}{2}} 
	\cbrac{\fint_{\Q} \modulus{\Upsilon u}^2\ d\mu}^{-\frac{1}{2}}.$
The case of $t \leq \tau$ is easy. 
So, suppose that $\tau \leq t \leq \delta^j$
and let $\xi$ be the Lipschitz function 
guaranteed in Proposition \ref{Prop:Harm:LipCube}
extended to $0$ outside of $\Q$.
and so write
$$
\modulus{\int_{\Q} \Upsilon u\ d\mu}
	\leq \modulus{\int_{\Q} (1 - \xi) \Upsilon u\ d\mu }
		+ \modulus{\int_{\Q} [\xi,\Upsilon] u\ d\mu}
		+ \modulus{\int_{\Q} \Upsilon(\xi u)\ d\mu}.$$
The last term is $0$ by (H7) and so we
are left with estimating the two remaining terms.
First, noting that 
$\spt (1 - \xi) \subset \tilde{\cE}_\tau$
we compute
\begin{align*}
\modulus{\int_{\Q} (1 - \xi) \Upsilon u\ d\mu }
	&\leq \modulus{\int_{\tilde{\cE}_\tau} (1 - \xi) \Upsilon u\ d\mu} 
	\leq \cbrac{ \int_{\tilde{\cE}_\tau} \modulus{\Upsilon u}^2\ d\mu}^{\frac{1}{2}}
		\mu(\tilde{\cE}_\tau) \\
	&\leq C_2^{\frac{1}{2}} \cbrac{\frac{a_0 \tau}{2 \delta^j}}^{\frac{\eta}{2}} 
		\mu(\Q)^{\frac{1}{2}} \cbrac{ \int_{\Q} \modulus{\Upsilon u}^2\ d\mu}^{\frac{1}{2}} \\
	&\leq C_2^{\frac{1}{2}} \cbrac{\frac{a_0 \tau}{2 t}}^{\frac{\eta}{2}}
		\mu(\Q)^{\frac{1}{2}} \cbrac{ \int_{\Q} \modulus{\Upsilon u}^2\ d\mu}^{\frac{1}{2}}.
\end{align*}
Now, for the second term. We note that
$\spt M_\xi \subset \spt \Lipp \xi \subset \tilde{\cE}_\tau$
and compute
\begin{multline*}
\modulus{ \int_{\Q} [\xi, \Upsilon] u }
	= \modulus{ \int_{\tilde{\cE}_\tau} M_{\xi}(x) u(x)\ d\mu(x)}
	\leq \cbrac{ \int_{\tilde{\cE}_\tau} \modulus{M_\xi}^2\ d\mu}^{\frac{1}{2}}
		\cbrac{ \int_{\tilde{\cE}_\tau} \modulus{u}^2\ d\mu}^{\frac{1}{2}} \\
	\leq \Lip \xi\ \mu(\tilde{\cE}_\tau)^{\frac{1}{2}} \cbrac{\int_{\Q} \modulus{u}^2}^{\frac{1}{2}} 
	\leq \frac{16}{a_0} C_2^{\frac{1}{2}} \cbrac{\frac{a_0 \tau}{2 t}}^{\frac{\eta}{2}}
		\frac{1}{\tau} \mu(\Q)^{\frac{1}{2}} \cbrac{\int_{\Q} \modulus{u}^2}^{\frac{1}{2}} \\
	\leq \frac{16}{a_0} C_2^{\frac{1}{2}} \cbrac{\frac{a_0 \tau}{2 t}}^{\frac{\eta}{2}}
		\mu(\Q)^{\frac{1}{2}} \cbrac{\int_{\Q} \modulus{\Upsilon u}^2}^{\frac{1}{2}}
\end{multline*}
where we have used Cauchy-Schwarz inequality to obtain the first inequality,
(H6) in the second, the condition (vi)
of Theorem \ref{Thm:Dya:Christ} in the third,
and substitution for $\frac{1}{\tau}$ in the last.
Combining these estimates,
we have
$$
\modulus{\int_{\Q} \Upsilon u\ d\mu}
	\leq D \frac{1}{t^{\frac{\eta}{2}}} \tau^{\frac{\eta}{2}} \mu(\Q)^{\frac{1}{2}}
		\cbrac{ \int_{\Q} \modulus{\Upsilon u}^2\ d\mu}^{\frac{1}{2}}$$
where 
$$D =C_2^{\frac{1}{2}} \cbrac{\frac{a_0}{2}}^{\frac{\eta}{2}} 
	+ \frac{16}{a_0} C_2^{\frac{1}{2}} \cbrac{\frac{a_0}{2}}^{\frac{\eta}{2}}
\quad\text{and}\quad 
\tilde{D} = C (2^\p C_1^\p a_0^{-\p})^{\frac{1}{2}}.$$
By Cauchy-Schwartz and multiplying both sides by $\mu(\Q)^{-2}$,
we find
$$
\modulus{\fint_{\Q} \Upsilon u\ d\mu}^2
	\leq 2 D^2 \frac{1}{t^\eta} \tau^{\eta} 
		\fint_{\Q} \modulus{\Upsilon u}^2\ d\mu.$$
The proof is complete by 
making a substitution for $\tau^\eta$.
\end{proof}

\begin{proposition}[Second term estimate]
\label{Prop:Harm:SecondHom}
For all $u \in \Hil$, we have
$$ \int_{0}^\infty \norm{\pri_t \Av_t(P_t - I)u}\ \frac{dt}{t} \lesssim \norm{u}^2.$$
\end{proposition}

Again, the proof of this proposition is omitted
since it resembles 
the proof of Proposition 5.7 in \cite{AKMC} with minor differences.

\subsection{Carleson measure estimate}

We begin this section with the following 
proposition which illustrates that 
the final term can be dealt with a 
Carleson measure estimate.

\begin{proposition}
For all $u \in \Hil$, we have
$$
\iint_{\Spa_+} \modulus{\Av_t u(x)}^2\ d\nu(x,t) \lesssim \norm{\nu}_{\Carl} \norm{u}^2$$
for every $\nu \in \Carl$.
\end{proposition}
\begin{proof}
First, we show that for almost every $x \in \Spa$,
$$
\Max^\ast \modulus{\Av_{\mdot} u}^2 (x) \lesssim \Max u(x)^2$$
where the constant depends only 
on $p$, $C_1$, $\delta$ and $a_0$.
Let $f \in \Lp[loc]{1}(\Spa_+, \C^N)$.
Then, we note that 
$$
\Max^\ast f(x) = \sup_{t > 0} \sup_{y \in B(x,t)} \modulus{f(y,t)}.$$
Fix $t$ such that $\delta^{j+1} < t \leq \delta^{j}$
and fix $x \in \union \DyQ_t$. Since
$\Av_tu(z) = 0$ when $z \not \in \union \DyQ_t$, 
take $y \in \union \DyQ_t$ such that 
$d(x,y) < t$. Let $Q \in \DyQ_t$ be the 
unique cube with $y \in Q$ and 
let $y_Q \in Q$ such that
$B(y_Q, a_0 \delta^j) \subset Q \subset B(y_Q, C_1\delta^j)$.
Then,
$d(y_Q,x) \leq d(y_Q,y) + d(y,x) \leq Ct$,
where $C = (C_1\delta^{-1} + 1)$.

Also
$$\mu (B(y_Q, Ct))
	\leq \mu (B(y_Q, C\delta^j))
 	\leq 2^p C^p a_0^{-p} \mu (B(y_Q, a_0 \delta^j))
	\leq 2^p C^p a_0^{-p} \mu(Q)$$
and therefore,
$$
\modulus{\Av_t u(y)} \leq \fint_{Q} \modulus{u}\ d\mu
	\leq 2^p C^p a_0^{-p} \fint_{B(y_Q,Ct)} \modulus{u}\ d\mu.$$
Moreover, 
$$\modulus{\Av_t u(y)}^2 \leq C' \cbrac{\fint_{B(y_Q,Ct)} \modulus{u}\ d\mu}^2$$
where $C' = 2^{2p} C^{2p} a_0^{-2p}.$ 

Now, since we have established that $x \in B(y_Q, Ct)$,
$$
\sup_{y \in B(x,t)} \modulus{\Av_tu(y)}^2 
	\leq C' \sup_{y \in B(x,t)} \cbrac{\fint_{B(y_Q(y),Ct)} \modulus{u}\ d\mu}^2
	\leq C' (\Max u(x))^2.$$

Let $\tilde{\Spa} = \intersect_{j} \union \DyQ^j$
and so
$
\mu(\Spa \setminus \tilde{\Spa})
	= \mu (\union_{j} \Spa \setminus \union \DyQ^j)
	\leq \sum_{j} \mu(\Spa \setminus \union \DyQ^j)
	= 0.$
Therefore, $x \in \tilde{\Spa}$, then $x \in \union \DyQ_t$ for all $t > 0$.
So, fix $x \in \tilde{\Spa}$. Then,
$$\Max^\ast \modulus{\Av_{\mdot} u}^2 (x)
	= \sup_{t > 0} \sup_{y \in B(x,t)} \modulus{\Av_t u(y)}^2 
	\leq C' \Max u(x)^2$$
which completes the proof.

Next, let $f(x,t) = \modulus{\Av_tu(x)}^2$.
Then,  
$
\norm{f}_{\Nfs}
	= \norm{\Max^\ast f}_1
	\lesssim \norm{\Max u}^2 < \infty$
by the Maximal Theorem \ref{Thm:Max:Max}. 
Invoking Carleson's Theorem 
\ref{Thm:Max:Carl} completes the proof.
\end{proof} 

Thus, to bound the final term,
it suffices to prove
$$A \mapsto \iint_{A} \modulus{\pri_t(x)}^2\ d\mu(x)\frac{dt}{t}$$
is a Carleson measure.
We follow \cite{AKMC}
and fix $\delta > 0$ to be chosen
later. Let 
$$K_\nu = \set{\nu' \in \bddlf(\C^N)\setminus\set{0}: 
\modulus{\frac{\nu'}{\modulus{\nu'}} - \nu} \leq \sigma}$$
and let $\cF$ be a finite set
of $\nu \in \bddlf(\C^N)$
with $\modulus{\nu}=1$
such that $\union_{\nu \in \cF} K_\nu = \bddlf(\C^N) \setminus \set{0}$.
We note as do the authors of \cite{AKMC} that 
it is enough to show 
$$\iint_{(x,t) \in \CBox_{\Q}, \pri_t \in K_\nu}
 \modulus{\pri_t(x)}^2\ d\mu(x)\frac{dt}{t} \lesssim \mu(\Q)$$
for each $\nu \in \cF$.
A stopping time argument allows us to reduce
this to the following.

\begin{proposition}
\label{Prop:Harm:Carl}
There exists a $0 < \beta < 1$ such that
for every dyadic cube $\Q \in \DyQ$
and $\nu \in \bddlf(\C^N)$ with 
$\modulus{\nu} = 1$, there exists
a collection $\set{\Q[k]} \subset \DyQ$
of disjoint subcubes of $\Q$ 
satisfying $\mu(E_{\Q, \nu}) > \beta \mu(\Q)$
and such that 
$$
\iint_{(x,t) \in E_{\Q,\nu}^\ast,\ \pri_t(x) \in K_\nu} 
	\modulus{\pri_t(x)}^2\ d\mu(x)\frac{dt}{t}
	\lesssim \mu(\Q)$$
where $E_{\Q,\nu} = \Q \setminus \union_{k} \Q[k]$
and $E_{\Q,\nu}^\ast = \CBox_{\Q}\setminus\union_{k} \CBox{\Q[k]}.$
\end{proposition}

We prove this via
defining a test function
similar to
the one found on page 484 in \cite{AKMC}.
Here, the authors
use a smooth cutoff function
in their construction.
Again, we rephrase this 
in terms of
a Lipschitz cutoff function
whose existence is guaranteed 
by the following lemma.

\begin{lemma}
\label{Lem:Harm:LipQ} 
Let $\Q \in \DyQ$. Then, there exists
a Lipschitz function $\eta: \Spa \to [0,1]$ such that
$\eta = 1$ on $B(x_{\Q},\tau C_1\len(\Q))$ and $\eta = 0$ on 
$\Spa \setminus B(x_{\Q},2\tau C_1\len(\Q))$ with
$$
\Lip \eta \leq \frac{4}{\tau C_1} \frac{1}{\len(\Q)}$$
whenever $\tau > 1$.
\end{lemma}
\begin{proof}
Fix $\Q \in \DyQ^j$, and
we have
$\Q \subset B(x_{\Q}, \tau C_1\delta^j) \subset B(x_{\Q}, 2\tau C_1\delta^j)$.
Also,
$$d(B(x_{\Q}, \tau C_1\delta^j), \Spa \setminus B(x_{\Q}, 2\tau C_1\delta^j))
	\geq (2\tau C_1 - \tau C_1)\delta^j = \tau C_1 \delta^j.$$

Now, we invoke 
Lemma 
\ref{Lem:Harm:LipSep} with 
$E = B(x_{\Q}, \tau C_1\delta^j)$ and $F = \Spa \setminus B(x_{\Q}, 2\tau C_1\delta^j)$
to find a Lipschitz $\eta: \Spa \to [0,1]$ with
$\eta = 1$ on $B(x_{\Q}, \tau C_1\delta^j)$, $\eta = 0$ on $\Spa \setminus B(x_{\Q}, \tau C_1\delta^j)$
and
$$
\Lip \eta \leq \frac{4}{d(B(x_{\Q}, \tau C_1\delta^j), \Spa \setminus B(x_{\Q}, 2\tau C_1\delta^j))}
	\leq \frac{4}{\tau C_1} \frac{1}{\delta^j}
	= \frac{4}{\tau C_1} \frac{1}{\len(\Q)}$$
which completes the proof.
\end{proof}

The test function is 
now defined as follows.
Let $\Q \in \DyQ$ and fix 
$\nu \in \bddlf(\C^N)$
with $\modulus{\nu} = 1$.
Let $\eta_{\Q}$
be the Lipschitz map guaranteed by
Lemma \ref{Lem:Harm:LipQ}
and let $w, \hat{w} \in \C^N$
such that $\adj{\nu}(\hat{w}) = w$
with $\modulus{w} = \modulus{\hat{w}} = 1$.
Furthermore, let $w_{\Q} = \eta_{\Q} w$ and
define
\begin{align*}
f^w_{\Q,\epsilon} &= w_{\Q} - \epsilon \len(\Q) \imath \Gamma
		(I + \epsilon\len(\Q)\imath \Pi_B)^{-1}w_{\Q}\\
	&= (1 + \epsilon \len(\Q)\imath \adj{\Gamma}_B)
		(1 + \epsilon \len(\Q)\imath \Pi_B)^{-1}w_{\Q}.
		\end{align*}
It is then an easy fact that
$\norm{w_{\Q}}^2 \leq
	(4\tau C_1a_0^{-1})^{\p} \mu(\Q)$
and we obtain the following lemma
analogous to Lemma 5.10 in \cite{AKMC}.

\begin{lemma}
\label{Lem:Harm:fest1}
There exists $c > 0$ such that for all $\epsilon > 0$,
$\norm{f^w_{\Q, \epsilon}} \leq c  \mu(\Q)^{\frac{1}{2}}$,
$\iint_{\CBox_{\Q}} \modulus{\Theta_t^B f^w_{\Q, \epsilon}}^2\ 
		d\mu(x) \dfrac{dt}{t}
		\leq c \dfrac{1}{\epsilon^2} \mu(\Q)$, and 
$\modulus{\fint_{\Q} f_{\Q,\epsilon}^w - w}\leq c \epsilon^{\frac{\eta}{2}}$.
\end{lemma}
\begin{proof}
The proof of the first two
estimates are essentially the 
same as that of Lemma 5.10 in \cite{AKMC}. 
To prove the last estimate,
note that since $\eta_{\Q} = 1$ on $\Q$,
we have on $\Q$ that
\begin{align*}
f_{\Q,\epsilon}^w - w 
	&=  w_{\Q} - \epsilon \len(\Q) 
		\imath(1 + \epsilon \len(\Q)\imath \Pi_B)^{-1}w_{\Q} - w \\
	&= (\eta_{\Q} - 1)w - \epsilon \len(\Q) 
		\imath(1 + \epsilon \len(\Q)\imath \Pi_B)^{-1}w_{\Q}  \\
	&= - \epsilon \len(\Q) 
		\imath(1 + \epsilon \len(\Q)\imath \Pi_B)^{-1}w_{\Q}.
\end{align*}

Setting $u = (1 + \epsilon\len(\Q)\imath \Pi_B)^{-1}w_{\Q}$
and $\Upsilon = \Gamma$,
we apply Lemma \ref{Lem:Harm:UpsilonHom}
\begin{align*}
\modulus{\fint_{\Q} f_{\Q,\epsilon}^w - w}
	&= \modulus{\fint_{\Q} \epsilon \len(\Q) 
		\imath(1 + \epsilon \len(\Q)\imath \Pi_B)^{-1}w_{\Q}} \\
	&= \epsilon\len(\Q) \modulus{\fint{\Q} 
		\imath(1 + \epsilon \len(\Q)\imath \Pi_B)^{-1}w_{\Q}} \\
	&\lesssim \frac{\epsilon\len(\Q)}{t^{\frac{\eta}{2}}}
		\cbrac{\fint_{\Q}\modulus{(1+\epsilon\len(\Q)\imath \Pi_B)^{-1}w_{\Q}}\ d\mu}^{\frac{\eta}{4}} \\
		&\qquad\qquad\qquad
		\cbrac{\fint_{\Q}\modulus{\Gamma(1 + \epsilon \len(\Q)\imath\Pi_B)^{-1}w_{\Q}}^2\ d\mu}^{\frac{1}{2}-\frac{\eta}{4}} \\
	&= \cbrac{\frac{\epsilon\len(\Q)}{t}}^{\frac{\eta}{2}}
		\cbrac{\fint_{\Q}\modulus{(1+\epsilon\len(\Q)\imath \Pi_B)^{-1}w_{\Q}}\ d\mu}^{\frac{\eta}{4}}  \\
		&\qquad\qquad\qquad
		\cbrac{\fint_{\Q}\modulus{\epsilon\len(\Q)\imath\Gamma
	(1 + \epsilon \len(\Q)\imath\Pi_B)^{-1}w_{\Q}}^2\ d\mu}^{\frac{1}{2}-\frac{\eta}{4}}. \\
\end{align*}
The proof is 
completed by
noting $t \simeq \len(\Q)$
and invoking  Proposition 2.5 and  Lemma 4.2 of \cite{AKMC}.
\end{proof}

The proof of Proposition \ref{Prop:Harm:Carl}
then follows a procedure 
similar to that 
which is used to prove  Lemma 5.12 in \cite{AKMC}.

We note that our hypotheses (H1)-(H8) remain
unchanged upon replacing $(\Gamma,B_1,B_2)$
by
$(\adj{\Gamma},B_2,B_1)$,
$(\adj{\Gamma},\adj{B_2}, \adj{B_1})$
and $(\Gamma, \adj{B_1}, \adj{B_2})$. 
Thus,
the hypothesis of Proposition
\ref{Prop:HypRes:Main}
is satisfied 
and Theorem \ref{Thm:HypRes:Main} is proved.

\bibliographystyle{amsplain}

\providecommand{\bysame}{\leavevmode\hbox to3em{\hrulefill}\thinspace}
\providecommand{\MR}{\relax\ifhmode\unskip\space\fi MR }
\providecommand{\MRhref}[2]{%
  \href{http://www.ams.org/mathscinet-getitem?mr=#1}{#2}
}
\providecommand{\href}[2]{#2}

\setlength{\parskip}{0mm}

\end{document}